\documentclass[letterpaper,12pt]{extarticle}

\usepackage[utf8]{inputenc}
\usepackage[margin= 0.9in]{geometry}
\usepackage{amsmath}
\usepackage{ dsfont }
\usepackage{ amssymb }

\newtheorem{theorem}{Theorem}[section]
\newtheorem{lemma}[theorem]{Lemma}

\newtheorem{remark}[theorem]{Remark}

\newenvironment{proof}[1][Proof]{\begin{trivlist}
\item[\hskip \labelsep {\bfseries #1}]}{\end{trivlist}}
\newenvironment{definition}[1][Definition]{\begin{trivlist}
\item[\hskip \labelsep {\bfseries #1}]}{\end{trivlist}}

\title{The double obstacle problem on non divergence form}
\author{Luis F. Duque}
\date{\today}

\begin{document}
\maketitle

\section{Abstract}
We study the regularity of the solution of the double obstacle problem form for fully non linear parabolic and elliptic operators. That is, we study a continuous function $u$ such that
\[ 
\left \{
  \begin{matrix}
  \phi_1 \leq u \leq \phi_2 \text{ on } Q_1 \\
  F(D^2 u) -\partial_t u \leq 0 \text{ on } \{ u < \phi_2 \} \cap Q_1 \\
  F(D^2 u) -\partial_t u \geq 0 \text{ on } \{ u > \phi_1 \} \cap Q_1
  \end{matrix}
\right .
\]

where $Q_1:=B_1x[-1,1]$ and $\phi_1$, $\phi_2$ are two uniformly separated obstacles. We show that when the obstacles are sufficiently regular $u$ is $C^{1,\alpha}$ in the interior of $Q_1$.

\section{Introduction}

The literature on elliptic single obstacle problems is vast. The optimal regularity of the solution and a detailed study of the free boundary can be found in \cite{Caffarelli} in the case of the Laplace operator. In \cite{Kinderlehrer}, Kinderlehrer studied the solution of this problem for elliptic operators with variable coefficients. The initial motivation of our work was precisely the generalization of Kinderlehrer's result to more general situations involving two obstacles and fully non linear elliptic and parabolic operators. \\

The regularity of solutions to the elliptic double obstacle problems in divergence form was studied on \cite{DalMaso} for the linear case. Later on, Kilpelainen and Ziemer (see \cite{Ziemer}) studied the Holder continuity of solutions for non-linear elliptic operators also in divergence form. \\

In \cite{Shahgholian} and \cite{Petrosyan}, Petrosyan and Shahgholian studied the regularity of the solution and the free boundary in non-divergence form of the parabolic single obstacle in different scenarios, including operators with constant coefficients and fully non linear-elliptic ones. They also presented the relation between this problems and the study of american options and choose their obstacles accordingly.\\

The main results of our paper are the interior $C^{1,\alpha}$ regularity of the solutions of both the elliptic and parabolic versions of this problem (see Theorem \ref{theorem_interior_regularity_elliptic} and Theorem \ref{C1InteriorRegularityParabolic})

The key step in this proofs is to study the way in which the solution grows away from the obstacle at the contact points (this is, the points in which the solution touches an obstacle). This growth is studied on Lemma \ref{growth_at_contact} and Lemma \ref{modulus_lemma_p} for the elliptic and the parabolic cases respectively. In the Appendix we sketch a proof of existence and regularity of solutions for the elliptic double obstacle problem when the obstacles are smooth, to do so we use a penalisation method as in \cite{Friedman}.

\section{Notation and basic definitions} 

\begin{itemize}
	\item $B_r(x_0)$ denotes a ball in space centered at $x_0$
	\item $Q_R(x_0, t_0):=B_R(x_0) \times [t_0-R^2, t_0+ R^2]$ 
	\item $Q^+_R(x_0, t_0):=B_R(x_0) \times [t_0, t_0+ R^2]$
	\item $Q^-_R(x_0, t_0):=B_R(x_0) \times [t_0-R^2, t_0]$ 
	\item $Q_R:=Q_R(0,0)$
	\item $Q^+_R:=Q^+_R(0,0)$
	\item $Q^-_R:=Q^-_R(0,0)$
	\item $\partial_p Q_R(x_0,t_0):= \partial B_R(x_0) \times(t_0-R^2, t_0 + R^2) \cup B_R(x_0)\times \{t_0-R^2 \}$
	\item $\partial_i u$ denotes the (spatial) derivative of $u$ in the $i$-th direction
	\item $\partial _t u$ denotes the derivative of $u$ the time direction 

\end{itemize}

\begin{definition}
 Let $S$ the space of al the $n\times n$ symmetric matrices. We say that $F:S\rightarrow \mathbb{R}$ is a uniformly elliptic operator if there are two constants $0<\lambda< \Lambda$ such that for every $N, M\in S$ with $N\geq 0$ we have 

\begin{equation} \label{ellipticity_definition}
\lambda||N|| \leq F(M+N)-F(N) \leq \Lambda ||N||
\end{equation}

In this case we denote $\lambda \leq F\leq \Lambda$
\end{definition}

\begin{remark}\label{Fully_non_linear_properties}
 We state some of the rescaling properties of fully non linear operators that will be used throught this paper. Let $F$ a uniformly elliptic operator $\lambda \leq F\leq \Lambda$.

\begin{enumerate}

\item Define $F_c:S\rightarrow \mathbb{R}$ such that $F_c(M):=\frac{1}{c}F(cM)$ for every $M\in S$ then $F_c$ has the same ellipticity of $F$, that is $\lambda \leq F_c\leq \Lambda$

\item Let $u$ such that $F(D^2 u)=0$ on $B_r$ and let $c\in \mathbb{R}^+$, define $\hat{u}:=cu$ on $B_r$ then $\hat{u}$ is the solution of a Fully non linear elliptic equation with the same ellipticity of $F$, that is $F_c(D^2\hat{u})=0$ on $B_r$.

\item Let $u$ such that $F(D^2 u)=0$ on $B_r$ and define $\hat{u}(x):=u(\frac{x}{r})$ on $B_1$ then $\hat{u}$ is the solution of a Fully non linear elliptic equation with the same ellipticity of $F$,  that is $F_{r^2}(D^2\hat{u})=0$ on $B_1$.

\item Let $\lambda \leq F_i\leq \Lambda$ a sequence of elliptic operators and $u_i\rightarrow u_0$ uniformly on $B_1$ such that $F_i(D^2u_i)=0$ on $B_1$ then $F_0(D^2u_0)\leq 0$ on $B_1$ for some $F_0$ uniformly elliptic with $\lambda \leq F_0\leq \Lambda$

\item Let $u$ such that $F(D^2u)-\partial_t u =0$ on $Q_1$, take $0<\lambda<1$ and $A>0$ and define $\hat{u}(x,t):=\frac{1}{A}u(\lambda x, \lambda^2 t)$ then $\hat{u}$ satisfies a fully non linear parabolic equation with the same ellipticity of $F$, that is $F_{\frac{A}{\lambda^2}}(D^2\hat{u})-\partial_t \hat{u} =0$ on $Q_1$

\item Let $\lambda \leq F_i\leq \Lambda$ a sequence of elliptic operators and $u_i\rightarrow u_0$ uniformly on $Q_1$ such that $F_i(D^2u_i)-\partial_t u_i=0$ on $B_1$ then $F_0(D^2u_0)-\partial_t u_0\leq 0$ on $Q_1$ for some $F_0$ uniformly elliptic with $\lambda \leq F_0\leq \Lambda$

\end{enumerate}
\end{remark}

\section{Elliptic Double Obstacle problem} \label{Section_Elliptic_Problem}
\subsection{Statement of the Problem} 

Let $\phi_1, \phi_2$ continuous and uniformly separated functions on $\bar{B_1}$, $ g \in C(\partial B_1)$ compatible with $\phi_1$ and $\phi_2$ (that is $\phi_1 \leq g \leq \phi_2  $ on $\partial B_1$) and an elliptic operator $F$ with $0<\lambda \leq F\leq \Lambda$ . We say that a continous function $u$ is a solution to the elliptic double obstacle problem ($\phi_1, \phi_2$, $F$, $B_1$) if

\[ \label{elliptic_problem}
\left \{
  \begin{matrix}
  \text{ u is continuous on } \overline{B}_1  \\
  \phi_1 \leq u \leq \phi_2 \text{ on } \overline{B}_1 \\
  u=g \text{ on } \partial B_1 \\
  F(D^2 u) \leq 0 \text{ on } \{ u < \phi_2 \} \cap B_1 \\
  F(D^2 u) \geq 0 \text{ on } \{ u > \phi_1 \} \cap B_1
  \end{matrix}
\right .
\]

For such a $u$ we define the following subets of $B_1$: \\

$E_1:=\{u = \phi_1\}$, \: $E_2:=\{u = \phi_2\} $, \: and \:$E:=E_1 \cup E_2$ \: are closeds set, and we call them our contact regions\\

$A_1:= \{ u> \phi_1\}$ , \: $A_2:= \{ u< \phi_2\}$ \: and \: $A:= \{\phi_1 <u< \phi_2\}$ are open sets known as the non contact regions\\

$\Gamma_1:= \partial A_1 $,\: $\Gamma_2:= \partial A_2 $\: and \: $\Gamma:= \Gamma_1 \cup \Gamma_2$. $\Gamma$ is known as the free boundary.

\subsection{Regularity of the solution: Elliptic case}

The following lemma will be a recurrent tool 

\begin{lemma}\label{epsilon_lemma} ($L^\epsilon _w$) Let $u$ non-negative super solution on $B_{r}$, that is $F(D^2 u)\leq 0$ then
 
\begin{equation} 
||u||_{L^\epsilon_w(B_{r/2})} \leq C u(0) r^n
\end{equation} 

or, equivalently

\begin{equation} \label{equation_epsilon_lemma}
|\{u>N\}\cap B_{r/2}| \leq C \frac{u(0) r^n}{N^\epsilon} \text{\:\:\: for any \:\:\:} N>0
\end{equation}

Where \: $C, \epsilon>0$ \: are universal constants 
\end{lemma}
\begin{proof}
 Lemma 4.5. on \cite{Cabre} $\square$.
\end{proof}

Notice that if a non-negative function $u$ satisfies a mean value property (for instance when $F=\Delta$) Lemma \ref{epsilon_lemma} follows immediately (with $\epsilon = 1$) since

\begin{equation}
u(0)=\frac{1}{|B_{r/2}|}\int_{B_{r/2}}u \geq \frac{1}{|B_{r/2}|}\int_{B_{r/2}\cap\{ u \geq N\}}u \geq \frac{N|B_{r/2}\cap\{ u \geq N\}|}{|B_{r/2}|}  
\end{equation}

\begin{remark}\label{maxminremark}
If $u$ is a viscosity solution of the double obstacle problem ($\phi_1, \phi_2$, $F$, $B_1$) and $\gamma$ is a constant such that $\phi_1<\gamma<\phi_2$ then $w:=\max(u, \gamma)$ is a subsolution of $F$, that is $F(D^2w)\geq 0$ on $B_1$. Similarly, $\min(u, \gamma)$ is a supersolution of $F$
\end{remark}

\begin{lemma} \label{growth_at_contact} (growth near contact points) Let $u$ a solution to the double obstacle problem $(\phi_1, \phi_2, F, B_1)$, and let $\sigma(r)$ the modulus of continuity of the obstacles. Suppose moreover that $x_0\in E_1$ (i.e. $u(x_0)=\phi_1(x_0)$) then $u \leq \phi_1(x_0) + C\sigma(r)$ on $B_{r/16}(x_0)$, where $C>0$ is a universal constant. 
\end{lemma}
\begin{proof}

Without loss of generality, suppose that $x_0=0$,  $u(0)=\phi_1(0)=\sigma(r)$ and $u \geq 0$ on $B_r(x_0)$ (if this is not the case, we could translate and study instead $\overline{u}:= u-(\phi_1(0) - \sigma(r))\geq 0$ on $B_r(x_0)$ instead of $u$). We consider the following cases: \\

\textbf{Case 1:} If $u$ touches the upper obstacle at $x_1 \in B_{r/4}$, we claim that $u(x_1)\leq  M\sigma(r)$ for some universal constant $M$. This implies that $u \leq (M+1)\sigma(r)$ on $B_{r/2}$. Suppose by contradiction that $u(x_1) = M\sigma(r)$ for some universal constant $M>0$ very big (to be chosen), we have then that $u_1:=(M+1) \sigma(r) - u\geq 0$ on $B_{r}(x_1)$, also as $u_1(x_1)\leq\sigma(r)$ we can apply Lemma \ref{epsilon_lemma} to $\min(u_1, M\sigma(r))$ (see Remark \ref{maxminremark}), and as $B_{r/4}(x_1) \subset B_{r/2}(x_1)$ we get

\begin{equation}\label{epsilon1}
\begin{split}
|\{  u < \frac{M\sigma(r)}{2} \}  \cap B_{r/4}(x_1)| & \leq \frac{C r^n \sigma(r)}{\left((M+1)\sigma(r)-\frac{M \sigma(r)}{2}\right)^\epsilon} \\
& \leq \frac{C r^n \sigma(r)^{(1-\epsilon)}}{\left(\frac{M}{2}\right)^\epsilon} \leq
\frac{Cr^n}{M^{\epsilon}} 
\end{split}
\end{equation}

Also, as we can apply lemma \ref{epsilon_lemma} to $\min(u, M \sigma(r))$ on $B_{r}$ (see Remark \ref{maxminremark}), and as $B_{r/4}(x_1) \subset B_{r/2}$ we get

\begin{equation}\label{epsilon2}
|\{ u \geq \frac{M\sigma(r)}{2} \}  \cap B_{r/4}(x_1)| \leq \frac{C r^n \sigma(r)}{\left( \frac{M\sigma(r)}{2} \right)^\epsilon} \leq \frac{Cr^n}{M^{\epsilon}} 
\end{equation}

where $C$, $\epsilon>0$ are universal constants. And hence, from equations \ref{epsilon1} and \ref{epsilon2}, we can pick  a universal $M>0$, not depending on $r$ so that 

$$|\{  u > \frac{M\sigma(r)}{2} \}  \cap B_{r/4}(x_1)| + |\{ u \leq \frac{M\sigma(r)}{2} \}  \cap B_{r/4}(x_1)| < |B_{r/4}(x_1)|$$

which is a contradiction, and we are done with Case 1. \\

\textbf{Case 2:} If $u$ does not touch the upper obstacle in $B_{r}(x_0)$, we claim that $u(x)\leq  M_0\sigma(r)$ on $B_{r/4}$ for some universal constant $M_0>0$.
  
From lemma \ref{epsilon_lemma}, as we are not touching the upper obstacle we have  
\begin{equation} \label{tempequ1}
||u||_{L^\epsilon_w(B_{r/2})} \leq C \sigma(r) r^n
\end{equation}

Let $\overline{u}:= \max(u,2\sigma(r))$, we know thenn $F(D^2 \overline{u}) \geq 0$ on $B_{r/4}$ (as $u$ is not touching the upper obstacle in this region) and from the previous equation we get that $||\overline{u}||_{L^\epsilon_w(B_{r/4})} \leq C \sigma(r) r^n$. That is, we know $\overline{u}$ is a subsolution of $F$ that is also on $L^{\epsilon}$ so Lemma 4.4. on \cite{Cabre} gives us that $u$ is bounded on the interior, moreover, $u\leq C\sigma(r)$ as desired $\square$
\end{proof}

Now that we have a growth estimate of our solution on the contact points (Lemma \ref{growth_at_contact}), the regularity of $u$ in the interior of $B_1$ will follow once we adapt some results from \cite{Kinderlehrer} to our situation, the rest of the section focuses on doing this.

\begin{lemma}\label{auxiliar_holder_growth_0} Let $U\subset\mathbb{R}^n$ open and bounded, $\rho_0>0$ a constant and $h:\bar{U}\rightarrow \mathbb{R}$ continuous with $F(D^2 h)=0$ on $U$ such that 
	$$ \sup_{B_\rho(x_0)\cap U} |h(x)-h(x_0)|\leq \sigma(\rho) \text{\:\:\: for \:\:\:} x_0\in \partial U  \text{\:\:\: and \:\:\:}  0<\rho\leq \rho_0 $$

then \: $|h(x)-h(x')|\leq \sigma(|x-x'|)$ \: for any \: $x, x'\in U$ \: with \: $|x-x'|\leq \rho_0$. 
\end{lemma}

\begin{proof}
Let $e\in\partial B_1$, for $0<\rho\leq \rho_0$ consider the function 

$$h_\rho(x):= h(x+e\rho) \text{\:\: on  \:\:} U_\rho:= \{ x-e\rho \:|\: x\in U\}$$

Notice that $F(D^2 h_\rho)=F(D^2 h)=0 \text{\:\: on  \:\:} U_\rho \cap U$ and hence

\begin{equation} \label{Pucci_Equation}
 M^-(D^2(h_\rho-h)) \leq F(D^2 h_\rho) - F(D^2 h) \leq M^+(D^2(h_\rho-h))  \text{\:\:\: on \:\:\:} U\cap U_\rho
\end{equation}

Where $M^-, M^+$ are the Pucci Operators with the same ellipticity of $F$ (see Chapter 2.2. on \cite{Cabre}), and hence from the comparison principle we get

$$\sup_{U\cap U_\rho} |h_\rho-h| \leq \sup_{\partial(U\cap U_\rho)} |h_\rho-h|\leq \sigma(\rho)$$ 

and the lemma follows $\square$.
\end{proof}

\begin{lemma}\label{auxiliar_holder_growth_1} Let $U\subset\mathbb{R}^n$ open and bounded, $\rho_0$ a positive constant and $h:\overline{B_1 \cap U}\rightarrow \mathbb{R}$ continuous with \:\: $F(D^2 h)=0$ \: on \: $B_1 \cap U$ such that

\begin{equation}\label{equation_growth_fb_1}
\sup_{B_\rho(x_0)\cap U} |h(x)-h(x_0)|\leq A\rho^\alpha \text{\:\:\: for \:\:\:} x_0\in B_1\cap\partial U  \text{\:\:\: and \:\:\:} 0<\rho\leq \rho_0
\end{equation}

then for $x,x'\in  B_{1-\delta}\cap\partial U \text{\:\:with\:\:} |x-x'|\leq \rho_0 $ we have

$$|h(x)-h(x')|\leq C\sigma(|x-x'|)+ c \frac{||h||_{L^\infty(B_1 \cap U)}}{\delta}|x-x'| $$

\end{lemma}

\begin{proof}
We want to apply Lemma \ref{auxiliar_holder_growth_0} to $h$ on $U\cap B_{1-\delta}$, that is, we need to control the growth of $h$ near $\partial (U\cap B_{1-\delta})$. Notice that the growth near $\partial U$  is already controlled by hypothesis (Equation \ref{equation_growth_fb_1}).\\

If $x\in U\cap \partial B_{1-\delta}$ and $d(x,\partial U)>\delta$ we have that $F(D^2 h)=0 \:\: on \:\: B_\delta(x)$ and hence if we reescale the $C^{1,\alpha}$ estimate for fully non linear elliptic equations (Corollary 5.7 on \cite{Cabre}) we get $||\nabla h||_{L^\infty(B_{\delta/2}(x))}\leq \frac{C}{\delta}||h||_{L^\infty(B_{\delta}(x))}$, and hence 

\begin{equation}\label{eq0_t}
|h(y)-h(x)|\leq \frac{C}{\delta}||h||_{L^\infty}|y-x| \text{\:\: for \: } y \in B_{\delta}(x)
\end{equation}

Finally, without loss of generality take $d(y, \partial U)=d(y, \overline{y}) \leq d(x, \partial U)=d(x,\overline{x})$ with $\overline{x}, \overline{y}\in \partial U$ let $d:=d(x,y)$ and $r:= d(x, \partial U)\leq \frac{\rho_0}{2}$. Consider the following two situation \\

First, if $d\leq \frac{r}{2}$ we have $F(D^2 (h-h(\overline{x})))=0 \:\: on \:\: B_r(x)$, so we can rescale the $C^{\alpha}$ regularity result for fully non linear elliptic equations (Proposition 4.10 on \cite{Cabre}) to get 

$$ \frac{|h(x)-h(y)|}{|x-y|^\alpha} \leq \frac{||h-h(\bar{x})||_{L^{\infty}(B_r(x))}}{r^\alpha}\leq \frac{A(2r)^\alpha}{r^\alpha}= 2^\alpha A$$

The second situation is when $d\geq \frac{r}{2}$, we have then
\begin{equation} \label{eq1_t}
\begin{split}
|h(x)-h(y)| & \leq |h(x)-h(\bar{x})|+|h(y)-h(\bar{y})| +|h(\bar{x})-h(\bar{y})| \\
 & \leq Ar^\alpha + A r^\alpha + A|\bar{x}- \bar{y}|^\alpha \\
 & \leq  2Ar^\alpha + A|2r + d|^\alpha \leq C d^\alpha = C|x-y|^\alpha
\end{split}
\end{equation}

for some universal constant $C$. The result follows once we apply Lemma \ref{auxiliar_holder_growth_0} together with Equation \ref{eq0_t} and Equation \ref{eq1_t} $\square$.
\end{proof}

\begin{theorem}\label{Holder_Regularity_u}($C^{\alpha}$ regularity)
 Let $\phi_1, \phi_2 \in C^{\alpha}$ with modulus of continuity $\sigma(r)$ then $u$ has modulus of continuity 
 
$$ \sigma_u(r)=C_1\sigma(cr)+C_2\frac{||\phi_1||_{L^\infty} + ||\phi_2||_{L^\infty}}{\delta} r \text{\:\:\: on \:\:\:} B_\delta:= \{x\in B_1 | d(x, \partial B_1)>\delta  \} $$
 
\end{theorem}

\begin{proof}
This follows from Lemma \ref{growth_at_contact} and Lemma \ref{auxiliar_holder_growth_1}
$\square$.
\end{proof}

We now want to study the modulus of continuity of the first derivatives of the solution $u$ provided that we know the modulus of continuity of the derivatives of the obstacles.

\begin{remark} \label{growth_from_tangent}
Let $\phi : B_1\rightarrow \mathbb{R}$. It can be shown using the fundamental theorem of calculus that  $\partial_e \phi$ has modulus of continity $\sigma(r)$ on every direction $e$ if and only if $\phi$ separates from its tangent plane in a $r\sigma(r)$ fashion, that is
$$ |\phi(B)-\phi(A) - (B-A) \cdot \nabla \phi(A)| \leq  |B-A|\sigma(|B-A|) \text{\:\: for any \:\:} A,B \in B_1$$
\end{remark}

\begin{lemma} \label{lemma_grow_away_obstacle_derivatives}
Let $u$ a solution to the elliptic double obstacle problem $(\phi_1, \phi_2, F, B_1)$ . Suppose that $\phi_1, \phi_2 \in C^{1,\alpha}(B_1)$, and  $\partial_e \phi_1, \partial_e \phi_2$ have modulus of continuity $\sigma(r)$ on every direction $e$. Then for every $x_0$ in the lower contact set $E_1\cap B_{1/2}$ we have
\begin{equation} \label{modulus_derivative_1_temporal}
\sup_{B_{r/2}(x_0)} (u-\phi_1)\leq Cr\sigma(r) \text{\:\: for \:\:} 0<r<r_0
\end{equation}

Similarly, if $x_0\in E_2$ we have 
\begin{equation}\label{modulus_derivative_2_temporal}
\sup_{B_{r/2}(x_0)} (\phi_2-u)\leq Cr\sigma(r) \text{\:\: for \:\:} 0<r<r_0
\end{equation}

Where $r_0>0$ is a universal constant.

\end{lemma}

\begin{proof}
Suppose, without loss of generality that $0\in E_1$, and let  $L(x):=\phi_1(0) + x \cdot \nabla \phi_1(0)$  (the first order Taylor expansion of $\phi_1$ at $0$). We want to show that
\begin{equation} \label{temp_growth_from_tangent}
U(x):=u(x)-L(x) \leq  C|x|\sigma(|x|) \text{\:\: on \:\:} B_{r_0}(0)
\end{equation}
For a universal constant $C>0$. 

To see this, notice that $U$ solves the elliptic double obstacle problem $(\Phi_1, \Phi_2, F, B_1)$ with $\Phi_1:=\phi_1- L$ and $ \Phi_2:=\phi_2-L$. Moreover, as our obstacles are uniformly separated (i.e. $\phi_2-\phi_1>\epsilon_0>0$), from Theorem \ref{Holder_Regularity_u} we know that our contact sets $E_1$ and $E_2$ are also uniformly separated, say $d(E_1, E_2) \geq 4 r_0$ for some $r_0>0$ and hence $U$ does not touch $\Phi_2$ on $B_{r_0}$. We know then that we are in a single obstacle type of situation on $B_{r_0}$ and hence we can proceed as Case 2 on Lemma \ref{growth_at_contact} (as the obstacle $\Phi_1$ has modulus of continuity $r\sigma(r)$ at $0$) to get Equation \ref{temp_growth_from_tangent}. Equation \ref{modulus_derivative_1_temporal} follows since
$$|u-\phi_1|(x)\leq |u-L|(x)+|\phi_1-L|(x)\leq C|x|\sigma(|x|).$$
Equation \ref{modulus_derivative_2_temporal} follows in an analog way $\square$.
\end{proof}

\begin{lemma}\label{holder_growth_to_interior}
Let $\phi_1, \phi_2 \in C^{1,\alpha}$ and $\sigma$ the modulus of continuity of $\partial_e \phi_1, \partial_e \phi_2$ on any direction $e$. Let $u$ a solution to the double obstacle problem $(\phi_1, \phi_2, F, B_1)$. Then for $x,x'\in  B_{1-\delta}\cap U \text{\:\:with\:\:} |x-x'|\leq r_0 $ we have

$$|u_e(x)-u_e(x')|\leq C\sigma(c|x-x'|)+ c \frac{||u||_{L^\infty}}{\delta^2}|x-x'|$$

\end{lemma}

\begin{proof}

Let $x_0 \in \Gamma_1$, as the obstacles are universaly separated and we already know that $u\in C^\alpha$ we can assume without loss of generality that $u < \phi_2$ on $B_{r_0}(x_0)$, let $x\in A \cap B_{r_0}(x_0)$ and $\hat{x}$ such that $d(x, E_1 \cap B_{r_0}(x_0))=d(\hat{x}, x)$

Notice that 
\begin{equation}\label{equation123}
 |u_e(\hat{x})-u_e(x_0)|  = |\partial_e\phi_1(\hat{x})-\partial_e\phi_1(x_0)| \leq  \sigma(|\hat{x}-x_0|) 
\end{equation}

Let $L$ the Taylor series expansion of $\phi_1$ at $\hat{x}$ then
\begin{equation}\label{equation234}
\begin{split}
 |u_e(x)-u_e(\hat{x})|& =  |(u-L)_e(x)-u_e(\hat{x}) + L_e(x)| =  |(u-L)_e(x)|\leq  \\
& \leq \frac{C||(u-L)||_{L^{\infty}(B_{|x-\hat{x}|}(x))}}{|x-\hat{x}|}  \leq \frac{C|x-\hat{x}|\sigma(|x-\hat{x}|)}{|x-\hat{x}|} = C\sigma(|x-\hat{x}|)   
\end{split}
\end{equation}

The last set of inequalities follows using that $ L_e(x)=\partial_e \phi_1(\hat{x})= u_e(\hat{x})$, interior estimates (since $F(D^2(u-L))=0$ on $B_{|x-\hat{x}|}(x)$) and Lemma \ref{lemma_grow_away_obstacle_derivatives}

From Equation \ref{equation123} and Equation \ref{equation234} it follows immediately that 
$$|u_e(x)-u_e(x_0)|  \leq  |u_e(x)-u_e(\hat{x})| + |u_e(\hat{x})-u_e(x_0)|\leq C\sigma(|x-x_0|)$$

Where $x_0$ is in the lower contact set, $|x-x_0|<r_0$ and $x$ is not on the contact region. We can now proceed exactly as in the proof of Lemma \ref{auxiliar_holder_growth_0} but applying the maximum principle to $u_e$ instead of $u$ to conclude the result $\square$.
\end{proof}

Finally we restate the last Lemma to get

\begin{theorem}\label{theorem_interior_regularity_elliptic} 
($C^{1, \alpha}$ regularity) If $\phi_1, \phi_2$ are $C^{1,\alpha}(B_1)$, and $u$ is a solution of the elliptic double obstacle problem $(\phi_1, \phi_2,F, B_1)$ then $u$ is $C^{1,\alpha}$ in the interior of $B_1$
\begin{proof}
It follows immediately from the previous Lemma $\square$.
\end{proof}

\end{theorem}

\section{Parabolic Double Obstacle problem}

\subsection{Statement of the problem}

In this section, let $\phi_1, \phi_2\in C(\overline{Q_1})$ such that $\phi_1 < \phi_2$. In this section our obstacles do not depend on time, this is $\phi_1(x,t)=\phi_1(x)$ and $\phi_2(x,t)=\phi_2(x)$; let $\sigma=\sigma(r)$ the modulus of continuity of both $\phi_1$ and $\phi_2$. Also, $ g \in C(\partial_p Q_1)$ will be the boundary data.\\

We say that $u\in C(\overline{Q_1})$ solves the parabolic double obstacle problem $(\phi_1, \phi_2, F, Q_1)$ if it satisfies the following

\[ 
\left \{
  \begin{matrix}
  \phi_1 \leq u \leq \phi_2 \text{ on } Q_1 \\
  u=g \text{ on } \partial_p Q_1 \\
  F(D^2 u) -\partial_t u \leq 0 \text{ on } \{ u < \phi_2 \} \cap Q_1 \\
  F(D^2 u) -\partial_t u \geq 0 \text{ on } \{ u > \phi_1 \} \cap Q_1
  \end{matrix}
\right .
\]

In the parabolic context, the definitions of $E_1, E_2, E, A_1, A_2, A, \Gamma_1, \Gamma_2, \Gamma$ are the same as in the elliptic problem, except that now we have subsets of $Q_1$ (instead of $B_1$) 

\begin{remark}\label{parabolicmaxminremark}
If $u$ is a viscosity solution of the parabolic double obstacle problem ($\phi_1, \phi_2$, $F$, $Q_1$) and $\gamma$ is a constant such that $\phi_1<\gamma<\phi_2$ then $w:=\max(u, \gamma)$ is a subsolution of $F$, that is $F(D^2w)-\partial_t w\geq 0$ on $B_1$. Similarly, $\min(u, \gamma)$ is a supersolution of $F$
\end{remark}

\subsection{Regularity of the solution: Parabolic case}\

The following is the analog of Lemma \ref{epsilon_lemma} to the parabolic situation

\begin{lemma}\label{epsilon_lemma_p} (parabolic-$L^\epsilon _ w$) Let $w \in C(Q_1)$ a non negative supersolution of the operator $F-\partial_t$  (i.e. $F(D^2 w) - \partial _t w \leq 0$), $0<R<1$ and $-1<t_0<0$ then there exist universal a universal constant $C>0$ such that

$$ ||w||_{L^\epsilon (\widehat{K}_{1})} \leq C w(0,0)$$
	
that is, 
		
\begin{equation}  \label{epsilon_lemma_p_equation} 
|\{w>N\}\cap \widehat{K}_{1}| \leq \left( \frac{C  w(0, 0)}{N} \right)^{\epsilon}  \text{\:\:\: for any \:\:\:} N>0
\end{equation} 	
		
where $\widehat{K}_{1}:= B_R \times (-1, t_0) $

\end{lemma}
\begin{proof}
See Theorem 4.5. on \cite{Imbert}. $\square$
\end{proof}

In order to motivate the previous lemma, consider the heat operator in one spatial dimension, that is, let $w$ satistying
\[ 
\left \{
  \begin{matrix}
   w'' -\partial_t w = 0  \text{ \: on \:}  Q_1 \\
  w\geq 0 \text{\: on \:}  Q_1
  \end{matrix}
\right .
\]

From Theorem 3 section 2.3. on \cite{Evans} we know that $w$ satisfies a mean value formula, let $E=E(0,0,1)$ the heat ball centered at $(0,0)$ with radious $1$. Then

$$ w(0,0) = \frac{1}{4 (2)^n} \iint _{E} w(x,t) \frac{|x|^2}{t^2} dxdt          
\geq \frac{1}{4 (2)^n} \iint _{E \cap \{ w\geq N \}} w(x,t) \frac{|x|^2}{t^2} dxdt $$

Notice that for any $A \subset Q_1$ we have $\iint _A|x|^2 dxdt \geq \frac{|A|^3}{4}$, so it follows that

$$ w(0,0)\geq \frac{N}{4 (2)^n} \iint _{E \cap \{ w\geq N \}}|x|^2 dxdt \geq C N |E \cap \{ w\geq N \}|^3 $$

And hence Equation \ref{epsilon_lemma_p_equation} holds in this particular case if we pick  $ \widehat{K}_{1} \subset E$.

\begin{lemma}\label{auxiliar_lemma_1}Let $w: Q^-_1\rightarrow \mathbb{R}$ a non-negative function such that $F(D^2 w)-\partial_t w=0$ and $w(x_0, t_0)>0$ for some $t_0<0$ then $w(0,0)>0$
\end{lemma}
\begin{proof}
This follows by contradiction using Lemma \ref{epsilon_lemma_p} $\square$.
\end{proof}

\begin{lemma}\label{modulus_lemma_p} (Holder growth at contact points)Let $\sigma(r)= Ar^{\alpha}$ be the modulus of continuity of $\phi_1$ and $\phi_2$. $u$ solves the problem $(\phi_1,\phi_2,F,Q_2)$ with $\lambda \leq F \leq \Lambda$. Let $X_0\in E\cap Q_{1/2}$ a contact point. Then $u$ grows in a $C^{\beta}(Q_{1/2})$ away from $X_0$. That is, there exist a universal constant A such that
$$ |u(X)-u(X_0)| \leq A r^{\beta} \text{ for } X\in Q_r(X_0) \text{ and } 0<r<1/2 $$

\end{lemma}
\begin{proof}
We consider the situation in which $X_0 \in E_1 \cap Q_{1/2}$, the other case is analog. We argue  by contradiction: if this result did not hold we would have two sequences of positive real numbers $\{ A_i \}_{i}$ and $\{r_i\}_i$ satisfying $\lim_{i \rightarrow \infty} A_i=\infty$ and $\lim_{i \rightarrow \infty} r_i=0$  and for each $i\in \mathbb{N}$ we  would have:
\[ 
\left \{
  \begin{matrix}
  \text{Two obstacles $\widehat{\phi^i_1}<\widehat{\phi^i_2}$ on $Q_2$ with modulus of continuity $\sigma$}\\
  \text{ A solution $\widehat{u}_i$ of the problem $(\widehat{\phi^i_1},\widehat{\phi^i_2}, F, Q_2)$ with $\lambda\leq F\leq \Lambda$}  \\ 
  \text{ A point $Z_i \in Q_{1/2}$ so that  $\widehat{\phi^i_1}(Z_i)= \widehat{u}_i(Z_i)$}  \\
  \text{ $|\widehat{u}_i(X)-\widehat{u}_i(Z_i)| \leq A_ir_i^{\beta}$ for $X\in Q_{r_i}(Z_i)$}  \\
  \text{A point $X_i \in Q_{\delta r_i}(Z_i)$ so that  $|\widehat{u}_i(\widehat{X}_i)-\widehat{u}_i(Z_i)| > A_i(\delta r_i)^{\beta}$ }
    \end{matrix}
\right .
\]

Where $\delta,\beta>0$ are to be chosen. We asume that $Z_i=(0,0)$ and define:

\begin{equation}\label{rescaled_solution}
	u_i(X):= \frac{ \widehat{u}_i( r_i x, r_i^2t)}{ A_ir_i^{\beta}} \text{ \: \: for \: $X=(x,t)\in Q_1$ }
\end{equation}

\begin{equation}\label{rescaled_lower_obstacle}
	\phi^i_1(X):= \frac{ \widehat{\phi^i_1}( r_i x, r_i^2t)}{A_ir_i^{\beta}} \text{ \: \: for \: $X=(x,t)\in Q_1$ }
\end{equation}

\begin{equation} \label{rescaled_upper_obstacle}
	\phi^i_2(X):= \frac{ \widehat{\phi^i_2}( r_i x, r_i^2t)}{ A_ir_i^{\beta}} \text{ \: \: for \: $X=(x,t)\in Q_1$ }
\end{equation}

Notice for instance that if we take $\beta<\alpha$ we will get
\begin{equation}\label{phi_oscilation}
|\phi^i_1(x)-\phi^i_1(0)|\leq \frac{\sigma(r_i)}{A_ir_i^{\beta}} = \frac{A r_i^{\alpha -\beta}}{ A_i}\rightarrow 0\text{\: \: on \: \:} Q_1
\end{equation}

Moreover, up to a subsequence we have the following:
\begin{enumerate}
\item $u_i$ solves a parabolic double obstacle problem $(\phi_1,\phi_2,F_i,Q_1)$ for some elliptic operator $\lambda \leq F_i\leq \Lambda$ (See Remark \ref{Fully_non_linear_properties}) \label{item:1}
\item $|| \phi^i_1-\phi^i_1(0)||_{L^\infty(B_1)}$, $|| \phi^i_2-\phi^i_2(0)||_{L^\infty(B_1)} \leq \frac{1}{i}$ (see Equation \ref{phi_oscilation}) \label{item:2}
\item $|u_i|\leq 1$ on $Q_1$ \label{item:3}
\item $u_i(X_i)> \delta ^ \beta$ for some $X_i \in Q_{\delta}$ \label{item:4}
\item $u_i(0)=\phi^i_1(0)=0$ \label{item:5}

\end{enumerate}

One of the following situations will occur: \\ 

\textbf{Case 1:}
There is a subsequence of $\{ u_i\}_i$ such that each $u_i$ touches both obstacles on $Q_{2\delta}$.  

The following reasoning the same as on Lemma \ref{growth_at_contact} and consists on applying twice the parabolic $L^\epsilon$ lemmma to reach a contradiction. 
Let $Y_i:=(y_i, t_i)$ and $\hat{Y_i} := (\hat{y_i}, \hat{t_i})$ in $Q_{2\delta}$ such that $\phi^i_1(Y_i)=u_i(Y_i)$ and $\phi^i_2(Y_i)=u_i(Y_i)$. Notice that $w:=u_i-u_i(Y_i)+\frac{1}{i}\geq 0$ on $Q_1$ so we can apply Lemma \ref{epsilon_lemma_p} and Remark \ref{parabolicmaxminremark} to $w$ at $Y_i$ and get

\begin{equation}\label{temp_epsilon1_p}
|\{ u_i > M \}  \cap  A_1  | \leq \left(\frac{C w(Y_i)}{M}\right)^\epsilon  \leq \left(\frac{C}{iM}\right)^\epsilon
\end{equation}

Where $A_1 := B_{\frac{1}{4}}(y_i) \times (t_i-\frac{1}{4}, t_i-\frac{1}{8})$

Notice also that $u_i\leq u_i(\hat{Y_i})+\frac{1}{i}$ on $Q_1$ so we can apply apply Lemma \ref{epsilon_lemma_p} to $h:= u_i(\hat{Y_i})+\frac{1}{i} -u_i$ (see Remark \ref{parabolicmaxminremark}) at $\hat{Y_i}$ and get

\begin{equation}\label{temp_epsilon2_p}
\begin{split}
|\{  u_i \leq M \}  \cap A_2|&\leq  \left( \frac{C h(\hat{Y_i}) }{ u_i(\hat{Y_i})+\frac{1}{i}-M }\right)^\epsilon\\ 
  & \leq  \left( \frac{C }{i (u_i(\hat{Y_i})+\frac{1}{i}-M)}\right)^\epsilon\\
  &\leq  \left( \frac{C }{i (1-M)}\right)^\epsilon 
\end{split}  
\end{equation}

Where $A_2 := B_{\frac{1}{4}}(\hat{y_i}) \times (\hat{t_i}-\frac{1}{4}, \hat{t_i}-\frac{1}{8})$

Notice that if we pick $\delta>0$ small enough we would have that  $|A_1\cap A_2|\geq c>0$ for some universal constant $c$ and hence
$$0<c\leq|A_1\cap A_2| =  \{  u_i > M \}\cap (A_1  \cap A_2)| +  \{  u_i \leq M \}\cap (A_1  \cap A_2)|$$

and if we pick $0 < M < \frac{\delta ^ \beta}{2}$ on Equation \ref{temp_epsilon1_p} and Equation \ref{temp_epsilon2_p}, use the last equation and pass to the limit when $i\rightarrow \infty$ we get
$$0 < c \leq  \left(\frac{C}{iM}\right)^\epsilon + \left( \frac{C }{i (1-M)}\right)^\epsilon \rightarrow 0$$
which is a contradiction.\\

\textbf{Case 2:} there is a subsequence of $\{ u_i\}_i$ of functions that does not touch the upper obstacle on $Q_{2\delta}$.

In this situation we follow \cite{Shahgholian} to get a contradiction. Define $v_i, \widehat{v}_i: Q_1 \rightarrow \mathbb{R} $ such that

\[ 
\left \{
  \begin{matrix}
  F(D^2 v_i)-\partial_t v_i = 0 \text{\: \: on \: \: } Q_{2\delta} \\
  v_i = u_i \text{ \: \:  on\: \: } \partial_p Q_{2\delta}
  \end{matrix}
\right .
\]

\[ 
\left \{
  \begin{matrix}
  F(D^2 \widehat{v}_i)- \partial_t \widehat{v}_i = 0 \text{\: \: on \: \:} Q_{2\delta} \\
  \widehat{v}_i = \max(u_i,\frac{1}{i}) \text{\: \: on \: \:} \partial_p Q_{2\delta}
  \end{matrix}
\right .
\]
From the comparison principle and as $u_i$ is the smallest super solution above $\phi^i_1$ we actually have  $-\frac{1}{i}\leq v_i \leq u_i \leq \widehat{v}_i$ on $Q_{2\delta}$.
From the interior estimates for fully non linear parabolic equations (see \cite{Wang}) and Arsela-Acoli it follows that $v_i\rightarrow v_0$ uniformly on $Q_{\delta}$.  We also know that $F(D^2v_0)-\partial_t v_0 \leq 0$ on $Q_{\delta}$ (see Remark \ref{Fully_non_linear_properties}) and as $ -\frac{1}{i}\leq v_i(0)\leq u(0)=\phi^i_1(0)=0$ we have $v_0(0)=0$ and hence, as $v_0\geq 0$ we conclude (using Lemma \ref{epsilon_lemma_p}) that $v_0=0$ on $Q^-_{\delta}$ .

Notice also that from the interior estimates for parabolic equations and from Arsela-Ascoli we have $\hat{v}_i\rightarrow \hat{v}_0$ uniformly on $Q_{\delta}$ and as $0\leq (\widehat{v}_i-v_i)\leq \frac{2}{i}$ on $\partial_p Q_{2\delta}$ (from the maximum principle) we know $0\leq (\widehat{v}_i-v_i)\leq \frac{2}{i}$ on $Q_{2\delta}$ and hence $\widehat{v}_0 = u_0 = 0$ on $Q^-_{\delta}$ and, as $u_i\leq \hat{v_i}$ 

At this point we know that $v_0$ is caloric, $v_0\leq 1$ on $Q_{2\delta}$ and $v_0=0$ on $Q_{\delta}$ so we can pick  $C,d,e>0$ properly so that the barrier $\psi:= C|x|^2+d t+e$ on  $Q^+_\delta$ is supercaloric, $\psi \geq 1$ on $\partial B_\delta \times (0,\delta^2)$ and $\phi(0)=0$ and hence if we redefine $\delta, \beta>0$ properly this contradicts the fact that $u_i(X_i)> \delta ^ \beta$ for some $X_i \in Q_\delta$ $\square$.

\end{proof}

We now bring the previous estimates to the interior of $Q_1$.

\begin{theorem}\label{theorem_interior_regularity_parabolic} 
($C^{\alpha}$ regularity)  Let $u$ a solution of the parabolic double obstacle problem ($\phi_1$, $\phi_2$, $F$, $Q^-_1$) with $\phi_1, \phi_2 \in C^\alpha$. Then $u$ is Holder continuous in $Q^-_{1-\delta}$ for any $0<\delta <1$
\end{theorem}
\begin{proof}
This proof follows exactly as Theorem \ref{Holder_Regularity_u} but using Lemma \ref{modulus_lemma_p} to control the growth of the $u$ near the free boundary, the parabolic maximum principle and the parabolic interior estimates (see \cite{Wang}) instead of its elliptic counterparts $\square$. 
\end{proof}

\begin{lemma} \label{modulus_lemma_derivative_p} 
 Let $u$ a solution of the parabolic double obstacle problem ($\phi_1$, $\phi_2$, $F$, $Q^-_1$), suppose $\partial _e\phi_1$ and $\partial_e \phi_2$ have modulus of continuity $\sigma(r)=Ar^{\alpha}$ on any direction $e$ on space-time, and let $X_0$ a contact point of the lower obstacle (that is $X_0\in E_1 \cap Q^-_{1/2}$) then 

\begin{equation} 
\sup_{Q_{r/2}(x_0)} (u-\phi_1)\leq Cr^{1+\alpha}
\end{equation}

Similarly, if $x_0\in E_2  \cap Q_{1/2}$ then 

\begin{equation}
\sup_{B_{r/2}(x_0)} (\phi_2-u)\leq Cr^{1+\alpha}
\end{equation}

\end{lemma}

\begin{proof}
We know from Remark \ref{growth_from_tangent} that if $\partial_e \phi_1$ has modulus of continuity $\sigma$ then
$$ |\phi(B)-\phi(A) - (B-A) \cdot \nabla \phi(A)| \leq  |B-A|\sigma(|B-A|) $$

Let $L_x$ the first order Taylor expansion of $\phi_1$ at $x$, that is

$$ L_x(y):=\phi_1(x)+(y-x)\cdot \nabla \phi_1(x) $$

We now want to show that if $x\in E_1\cap Q_{1/2}$ then

$$ U_x(y):=u(y)-L_x(y)\leq  A|y-x|\sigma(|y-x|) \leq A|y-x|^{1+\alpha}  \text{\:\: for \:\:} y\in Q_{1}(x) $$

For a universal constant $A>0$. 

Notice that $U_x$ solves the parabolic double obstacle problem $(\phi_1-L_x, \phi_2-L_x, F, Q_1^-)$. We point out that as the obstacles are uniformly separated and using the previous theorem we know that locally (i.e. around each point on $E_1 \cap Q^-_{1/2}$) our situation reduces to that of a single obstacle problem  (with obstacle $\phi_1-L_x$), in particular we can apply Lemma \ref{modulus_lemma_p} to $U_x$ and conclude the result $\square$.

\end{proof}

\begin{theorem}\label{C1InteriorRegularityParabolic}
($C^{1, \alpha}$ regularity) Let $u$ a solution of the parabolic double obstacle problem ($\phi_1$, $\phi_2$, $F$, $Q^-_1$) with $\phi_1, \phi_2 \in C^{1,\alpha}(Q^-_1)$ and $e$ any direction in space and time. Then $u_e$ is Holder continuous in $Q^-_{1-\delta}$ for any $0<\delta <1$.
\end{theorem}
\begin{proof}
This proof follows exactly as Theorem \ref{holder_growth_to_interior} but using Lemma \ref{modulus_lemma_derivative_p} to control the growth of $u$ near the free boundary, the parabolic maximum principle and the parabolic interior estimates (see \cite{Wang}) instead of its elliptic counterparts $\square$. 
\end{proof}

\section{Appendix: Existence and regularity of solutions when the obstacles are smooth}

In this section we sketch a proof of existence  of viscosity solutions to the elliptic problem from Section \ref{Section_Elliptic_Problem} in the case in which $\phi_1$ and $\phi_2$ are smooth. \\

Let $\phi_1, \phi_2\in C(\overline{B_r})$ such that $\phi_1 < \phi_2$, $ g \in C(\partial B_r)$  boundary data (with $\phi_1 \leq g \leq \phi_2  $ on $\partial B_r$ ) and an elliptic operator $0<\lambda \leq F\leq \Lambda$ . Our goal is to find \: $u:\overline{B_r} \rightarrow \mathds{R}$ \: that solves the elliptic double obstacle ($\phi_1, \phi_2$, $F$, $B_r$), that is:

\begin{equation} \label{Elliptic_Problem}
\begin{cases}
\begin{split}
  &\text{ u is continuous on } \overline{B}_r  \\
  &\phi_1 \leq u \leq \phi_2 \text{ on } \overline{B}_r \\
  &u=g \text{ on } \partial B_r \\
  &F(D^2 u) \leq 0 \text{ on } \{ u < \phi_2 \} \cap B_r \\
  &F(D^2 u) \geq 0 \text{ on } \{ u > \phi_1 \} \cap B_r
\end{split}
\end{cases}
\end{equation}

To do so we adapt the penalisation method presented in \cite{Friedman} to our situation. A similar approach can also be taken to prove the existence of solutions for the parabolic case. First we study the family of penalised equations

\begin{equation} \label{elliptic_penalized}
\begin{cases}
\begin{split}
F(D^2 u^\epsilon) &= \beta_\epsilon (u^\epsilon - \phi_1) - \beta_\epsilon (\phi_2 - u^\epsilon)  \quad &\textrm{on} \quad B_1 \\
u^\epsilon &= 0  \quad & \textrm{on} \quad \partial B_1
\end{split}
\end{cases}
\end{equation}

Where, for each $\epsilon > 0$,  $\beta _\epsilon:$ $\mathbb{R} \rightarrow \mathbb{R}$ is a smooth function with the following properties:

\begin{itemize}
  \item $\lim_{s\to - \infty}\beta_\epsilon(s)=-\infty$
  \item $\beta_\epsilon ':= \partial \beta_\epsilon > 0$ 
  \item $\beta_\epsilon(0)=-C$
  \item $-C \leq \beta_\epsilon(s)\leq C$ for every $s>0$
  \item $\lim_{\epsilon\to 0}\beta_\epsilon(s)= 0$ for every $s>0$ 
  \item $\lim_{\epsilon\to 0}\beta_\epsilon(s)= -\infty$ for every $s<0$ 
\end{itemize}

Our goal is to show that $u^\epsilon$ converges to the solution of our elliptic problem $u$ when $\epsilon \rightarrow 0$,  first we need some lemmas.

\vspace{5mm}

\begin{lemma} \label{smoothness_penalized_equation} 
The solutions $u^\epsilon$ to the penalised problem (Equation \ref{elliptic_penalized}) are  $C^{2, \alpha}$
\end{lemma}
\begin{proof}
Fix $\epsilon>0$, for each $N\in \mathbb{N}$ let $\beta_\epsilon ^N$ a truncation of $\beta _\epsilon$ between $N$ and $-N$ (i.e. $\beta_\epsilon ^ N := \max \{\min\{ \beta_\epsilon, N \}, -N \}$), regularize $\beta_\epsilon ^ N$ if necessary (i.e. change it a little bit so that it couples smoothly at levels $N$ and $-N$ and $\partial \beta_\epsilon ^N > 0$ still holds). \\

We define the mapping $T:W^{2,p}(B_1)\rightarrow W^{2,p}(B_1)$ with $T(w)=v$ where $v$ is the solution to the problem

\begin{equation} \label{penalized_equation_truncated}
\begin{cases}
\begin{split}
F(D^2v) &= \beta_\epsilon ^N (w - \phi_1) - \beta_\epsilon ^ N (\phi_2 - w)  & \textrm{on} &\quad B_1 \\
v &= g  & \quad \textrm{on} & \quad \partial B_1
\end{split}
\end{cases}
\end{equation}

This map is well defined since the right hand side of Equation \ref{penalized_equation_truncated} is in $L^p$ from our definition of $\beta_\epsilon^N$.\\

From the $W^{2,p}$ estimates for fully non linear equations (see \cite{Cabre}) we can find a radious $R>0$ so that $T(B_{R})\subset B_{R}$ and from Schauders fixed point theorem we get that there is an element $u^N \in W^{2,p}$ such that $T(u^N)=u^N$, that is:

\begin{equation} \label{penalized_equation_truncated}
\begin{cases}
\begin{split}
Lu^N &= - \beta_\epsilon ^N (u^N - \phi_1) + \beta_\epsilon ^ N (\phi_2 - u^N) & \textrm{on} &\quad B_1 \\
u^N &= g  & \quad \textrm{on} &\quad \partial B_1 \\
\end{split}
\end{cases}
\end{equation}

From the Sobolev embedding $W^{2,p} \hookrightarrow C^{1,\alpha}$ we get that the right hand side of Equation \ref{penalized_equation_truncated} is $C^\alpha$ and hence, from the Schauder estimates for fully non linear equations we have that $u^N\in C^{2+\alpha}$, that is $u^N$ has classic derivatives.

Notice at this point that, as $\epsilon > 0$ is fixed,  the right hand side of Equation \ref{penalized_equation_truncated} is bounded by a constant depending on $N$, we want to improve our bound so we dont have this dependency. \\

Lets first bound $\eta:=\beta_\epsilon ^N (u^N - \phi_1) - \beta_\epsilon ^ N(\phi_2 - u^N)$ from below on $B_1$. Suppose then that $\eta$ attains its minimum at $x_0 \in B_1$, there are 3 cases:

Case 1: $\phi_1(x_0)  \leq u^N(x_0)       \leq \phi_2(x_0)$ notice that $u^N(x_0) - \phi_1(x_0), \phi_2(x_0) - u^N(x_0)\geq 0$ and hence from the definition of $\beta_\epsilon ^N$ we have that $-C \leq |\beta_\epsilon^N (u^N(x_0) - \phi_1(x_0))|, |\beta_\epsilon ^ N(\phi_2(x_0) - u^N(x_0))| \leq C$ and hence $-2C \leq \eta$ on $\overline{Q}$ with $C$ as in the definition of $\beta_\epsilon$. \\

Case 2: If $\phi_1(x_0)  \leq \phi_2(x_0)  \leq u^N(x_0)$  we have $u^N(x_0) - \phi_1(x_0)\geq 0$ and $\phi_2(x_0) - u^N(x_0)\leq 0$, so from the properties of $\beta_\epsilon ^N$ we get $-C\leq \beta_\epsilon ^N (u^N(x_0) - \phi_1(x_0))$ and  $\beta_\epsilon ^N(\phi_2(x_0) - u^N(x_0))\leq 0$ and hence $-C \leq \eta$ on $\overline{B_1}$, again, with $C$ as in the definition of $\beta_\epsilon$. \\

Case 3: $u^N(x_0) \leq \phi_1(x_0)  \leq \phi_2(x_0)$. In this case notice that $u^N(x_0) - \phi_1(x_0)\leq 0$ and $\phi_2(x_0) - u^N(x_0)\geq 0$, so let $\overline{x}_0$ the point in which $\beta_\epsilon ^N(u^N-\phi_1)$ reaches its negative minimum, so $-C \leq \beta_\epsilon(\phi_2(\overline{x}_0) - u^N(\overline{x}_0)) \leq C$. As $\beta_\epsilon ^N$ is monotone,  $u^N-\phi_1$ attains its negative minimum at $\overline{x}_0$, and we have, 

\begin{equation} \label{last_temp_1}
 \begin{split}
  \beta_\epsilon ^N (u^N(\overline{x}_0) - \phi_1(\overline{x}_0))  
  	&=  F(D^2 u^N(\overline{x}_0)) + \beta_\epsilon ^N(\phi_2(\overline{x}_0) - u^N(\overline{x}_0))  \\
	&=  F(D^2 u^N(\overline{x}_0)) - F(D^2 \phi_1(\overline{x}_0)) +  F(D^2 \phi_1(\overline{x}_0)) + \beta_\epsilon ^N(\phi_2(\overline{x}_0) - u^N(\overline{x}_0))  \\
  &\geq  M^-(D^2(u^N-\phi_1)(\overline{x_0})) + F(D^2 \phi_1(\overline{x}_0))- C
\end{split} 
\end{equation}

On the last inequality we are using the definition of Pucci operators (see \cite{Cabre} or Equation \ref{Pucci_Equation}). Now from Equation \ref{last_temp_1}, and from our choice of $\overline{x_0}$ we get 

 \begin{equation}
 \beta_\epsilon ^N (u^N(\overline{x}_0) - \phi_1(\overline{x}_0))\geq C
 \end{equation}

And $C$ does not depend on $\epsilon$ or $N$. Now, notice that

 \begin{equation}
 	\beta_\epsilon ^N (u ^N(\overline{x}_0) - \phi_1(\overline{x}_0)) - \beta_\epsilon(\phi_2(x_0) - u^N(x_0)) \leq \eta(x_0)
\end{equation}

 and hence we can also bound $\eta$ from below as desired, so we are done with the this case. \\

We conclude then then that $\eta$ is bounded on $B_1$ from below by a constant not dependind on $\epsilon$ or $N$, we can proceed in an analog way to get a bound from above and hence we conclude that $\rVert F(D^2u^N)\lVert_{W^{2,p}}\leq C$ where $C$ is a constant not depending on $\epsilon$ or $N$ and hence $u^N$ is uniformly bounded by a constant not depending on $N$, so if we take $N$ sufficiently big we will have that $u^\epsilon := u^N$ is a smooth solution to our penalized equation (Equation \ref{elliptic_penalized}) $\square$.
\end{proof}

The following is just a slight improvement of the previous lemma, and its proof is indeed very similar.

\begin{lemma} \label{elliptic_beta_bound}
$|\beta_\epsilon (u^{\epsilon} - \phi_1) - \beta_\epsilon (\phi_2 - u^\epsilon)| \leq C$ on $\bar{B_1}$ for every $\epsilon>0$ where $C$ is a positive constant depending only on the ellipticity and the obstacles $\phi_1$ and $\phi_2$.
\end{lemma}
\begin{proof}

Let $v:=u_\epsilon$, define $G:=\beta_\epsilon (v - \phi_1) - \beta_\epsilon (\phi_2 - v)$ \: (where $G:\bar{B_1}\rightarrow \mathbb{R}$) \: and let $\bar{x}\in \bar{B_1}$ a point in wich $G$ achieves its maximum. 

We show that $G$ is bounded by above, observe the following:

\textbf{Case 1:} If $\bar{x} \in \partial B_1 $. Then, from the compatibility condition at the border we get $v(\bar{x})-\phi_1(\bar{x})\geq 0$  and $\phi_2(\bar{x})-v(\bar{x})\geq 0$ so from the definition of $\beta_\epsilon$ it follows immediately that $G(\bar{x})\leq 2C$ \\

\textbf{Case 2:} If $\bar{x} \in  B_1 $ and $\phi_1(\bar{x}) < v(\bar{x}) < \phi_2(\bar{x})$, it will follow immediate from the definition of $\beta_\epsilon$ that $G\leq 2C$   \\

\textbf{Case 3:} If $\bar{x} \in  B_1 $ and $\phi_2(\bar{x}) \leq v(\bar{x})$. As $\phi_1 \leq \phi_2$ we know that $(v (\bar{x} )- \phi_1 ( \bar{x}))\geq 0$ and hence $\beta_\epsilon(v (\bar{x} )- \phi_1 ( \bar{x}))\geq 0$. Let $\bar{\bar{x}} \in \bar{B_1}$ be the positive maximum of $- \beta_\epsilon (\phi_2 - v)$. If $\bar{\bar{x}}\in \partial B_1$ it follows immediately as before that $G\leq 2C$, so lets assume $\bar{\bar{x}}\in B_1$. From the monotonicity of $\beta_\epsilon$ we know $\bar{\bar{x}}$ is a minimum of $\phi_2 - v$ so the matrix $D^2(\phi_2 - v)(\bar{\bar{x}})$ is symmetric and non-negative. From the ellipticity of $F$ (Equation \ref{ellipticity_definition} with $M= D^2 v(\bar{\bar{x}})$ and $N=D^2(\phi_2 - v)(\bar{\bar{x}})\geq 0$) we have at $\bar{\bar{x}}$: 

\begin{equation}\label{temporal_case3_penalization_elliptic}
0 \leq \lambda || D^2(\phi_2 - v)|| \leq F(D^2 v+D^2(\phi_2 - v)) - F(D^2 v) = F(D^2\phi_2) - F(D^2 v) 
\end{equation}

So from Equation \ref{elliptic_penalized} and Equation \ref{temporal_case3_penalization_elliptic} we get

$$ F(D^2 v) = \beta_\epsilon (v - \phi_1) - \beta_\epsilon (\phi_2 - v) \leq F(D^2\phi_2) \text{\:\: at \:\:} \bar{\bar{x}}$$

that is

\begin{equation}\label{temporal_case3_penalization_elliptic_2}
- \beta_\epsilon (\phi_2 - v)_{\bar{\bar{x}}} \leq F(D^2\phi_2)_{\bar{\bar{x}}} - \beta_\epsilon (v - \phi_1)_{\bar{\bar{x}}}
\end{equation} 
 
and as $- \beta_\epsilon (\phi_2 - v)_{\bar{\bar{x}}}\geq 0$, $\beta_\epsilon$ is monotone and $\phi_1(\bar{\bar{x}})\leq\phi_2(\bar{\bar{x}})\leq v(\bar{\bar{x}})$ it follows that

\begin{equation}\label{temporal_case3_penalization_elliptic_3}
- \beta_\epsilon (\phi_2 - v)_{\bar{x}}\leq - \beta_\epsilon (\phi_2 - v)_{\bar{\bar{x}}} \leq F(D^2\phi_2)_{\bar{\bar{x}}} - C
\end{equation} 

This implies $G\leq C$ on $\bar{B_1}$ as desired, where $C$ is a constant not depending on $\epsilon$. Notice that if the maximum of the expression $- \beta_\epsilon (\phi_2 - v)$ is negative, the inequality $G\leq C$ on $\bar{B_1}$ is immediate.\\

\textbf{Case 4:} If $\bar{x} \in B_1 $ and $\phi_1(\bar{x}) \geq v(\bar{x})$. In this case, it is immeadiate that  $- \beta_\epsilon (\phi_2 - v)\leq C$ on $\bar{B_1}$, so let $\bar{\bar{x}}$ the maximum of $\beta_\epsilon (v - \phi_1)$. If this maximum is negative, there is nothing to prove, so lets suppose $\beta_\epsilon (v(\bar{\bar{x}}) - \phi_1(\bar{\bar{x}}))\geq 0$ and hence $v(\bar{\bar{x}}) \geq \phi_1(\bar{\bar{x}})$ so $D^2(v-\phi_1)_{\bar{\bar{x}}}\geq 0$ and we can use the ellipticity of $F$ as in the previous case to get

\begin{equation}\label{temporal_case4_penalization_elliptic}
0 \leq \lambda || - D^2(v-\phi_1)|| \leq  F(D^2\phi_1) - F(D^2 v)  \text{\:\:at \:\:} \bar{\bar{x}} 
\end{equation}

So

$$ F(D^2 v) = \beta_\epsilon (v - \phi_1) - \beta_\epsilon (\phi_2 - v) \leq F(D^2\phi_1) \text{\:\: at \:\:} \bar{\bar{x}}$$

And hence as $\beta_\epsilon$ is monotone and bounded by above we get

$$ \beta_\epsilon(v-\phi_1)_{\bar{x}} \leq  \beta_\epsilon(v-\phi_1)_{\bar{\bar{x}}}\leq F(D^2\phi_1)_{\bar{x}} + C$$
So $G$ is bounded by above in this case also.

Finding a bound of $G$ by below is similar to what we just did $\square$.
\end{proof}

\begin{theorem}
 If $\phi_1, \phi_2$ are smooth then the problem \ref{Elliptic_Problem} has a solution $u$ in the viscosity sense. Moreover $u\in C^{1,\alpha}$ for any $0<\alpha<1$
\end{theorem}
\begin{proof}
We fix $p>n$ arbitrary. From Lemma \ref{elliptic_beta_bound} and the $W^{2,p}$ estimates for fully non linear elliptic equations (see \cite{Cabre}) and compactness we have that up to a subsequence there exists $\widehat{u}$ so that $u_\epsilon \rightharpoonup \widehat{u}$ on $W^{2,p}(B_1)$ and from the Sobolev embeddings we have $u_\epsilon \rightarrow \widehat{u}$ on $C^{1,\alpha}(B_1)$ (strongly) where $0<\alpha=\alpha(p,n)<1$ coming from the Sobolev embedding.

From uniform convergence it follows that $\phi_1\leq \widehat{u}\leq \phi_2$, to see this we proceed by contradition. Suppose that $u(x)-\phi_1(x)=-\delta$ for some $x\in B_1$ and $\delta>0$. Then for $\epsilon>0$ sufficiently small we have $u_\epsilon(x)-\phi_1(x)<-\frac{\delta}{2}$ but then from the properties of $\beta_\epsilon$ we get 

$$ \lim_{\epsilon \rightarrow 0} \beta_\epsilon(u_\epsilon(x)-\phi_1(x)) \leq  \lim_{\epsilon \rightarrow 0} \beta_\epsilon(-\frac{\delta}{2})= \infty$$

And this contradicts Lemma \ref{elliptic_beta_bound}. If we had $\widehat{u}(x)>\phi_2(x)$ for some $x\in B_1$ we get to a similar contradiction.

We are only left to show that $\widehat{u}$ is indeed a solution to our problem. Let $x\in B_1$ such that $\phi_1(x)<\widehat{u}(x)<\phi_2(x)$, then, as $\widehat{u}\in C^\alpha$ we know that $\phi_1 <\widehat{u}< \phi_2$ on $B_\delta(x)$ for $\delta$ sufficiently small, moreover from uniform convergence we get  $\phi_1<u_\epsilon < \phi_2$ on $B_\delta(x)$ when $\epsilon>0$ is sufficiently small by redefining $\delta$. And hence, as $\lim_{\epsilon\to 0}\beta_\epsilon(s)= 0$ when $s>0$ we get $F(D^2 \widehat{u}(x))=0$.

If $\phi_1(x)=\widehat{u}(x)$ we have (as $\phi_1<\phi_2$) that $\lim_{\epsilon\rightarrow 0} \beta_\epsilon(\phi_2 - u_\epsilon) =0$ and hence $F(D^2 \widehat{u}(x))\leq 0$. The situation when $\widehat{u}(x)=\phi_2(x)$ follow in the same way $\square$.

\end{proof}

\section{Acknowledgements}
I would like to thank professor Luis Caffarelli for his valuable advice, support, for his enormous patience and for the many meetings in which we discussed material directly or indirecly related with this paper. I would also like to thank Dennis Kriventsov, Pablo Stinga (Iowa State University), Hui Yu (The University of Texas at Austin), Hernan Vivas (The University of Texas at Austin) and Xavier Ros-Oton (The University of Texas at Austin) for their support and the many discussions we had.  

Finally I would like to thank Colciencias and the National Science Fund for supporting this project.

\end{document}